\documentclass[a4paper]{amsart}
\usepackage{amssymb}
\usepackage{amsthm}
\usepackage{amsmath,amscd}
\usepackage[mathscr]{euscript}
\usepackage[all]{xy}
\usepackage[utf8]{inputenc}
\usepackage{lmodern}
\usepackage[T1]{fontenc}
\usepackage[textwidth=14.5cm,hcentering]{geometry}
\usepackage[pagebackref=true,breaklinks=true,letterpaper=true,colorlinks]{hyperref}
\swapnumbers

\theoremstyle{plain}
\newtheorem{theo}{Theorem}[section]
\newtheorem{lemm}[theo]{Lemma}
\newtheorem{prop}[theo]{Proposition}
\newtheorem{coro}[theo]{Corollary}

\theoremstyle{definition}
\newtheorem{defi}[theo]{Definition}
\newtheorem{nota}[theo]{Notation}

\theoremstyle{remark}
\newtheorem{rem}[theo]{Remark}

\numberwithin{equation}{section}
\newcommand{\op}{^{\mathrm{op}}}
\newcommand{\cat}{\mathbf}
\newcommand{\oper}{\mathcal}
\newcommand{\on}{\mathrm}
\newcommand{\Emb}{\on{Emb}}
\newcommand{\Map}{\on{Map}}
\newcommand{\Mod}{\cat{Mod}}

\newcommand{\oMod}{\oper{M}od}
\newcommand{\oMfld}{\oper{M}an}

\newcommand{\Mfld}{\cat{Man}}
\newcommand{\un}{\mathbb{I}}
\newcommand{\id}{\mathrm{id}}
\renewcommand{\S}{\cat{S}}
\newcommand{\Spec}{\cat{Spec}}
\newcommand{\goto}[1]{\stackrel{#1}{\longrightarrow}}
\newcommand{\HC}[1]{\mathrm{HH}_{#1}}
\newcommand{\HH}[1]{\mathrm{HH}^{#1}}

\renewcommand{\L}{\mathbb{L}}

\title[Factorization homology and calculus]{Factorization homology and calculus à la Kontsevich Soibelman}

\author{Geoffroy Horel}

\begin{document}

\address{Mathematisches Institut \\
Einsteinstrasse 62\\
D-48149 Münster\\
Deutschland}
\email{geoffroy.horel@gmail.com}
\thanks{The author was partially supported by an NSF grant.}

\begin{abstract}
We use factorization homology over manifolds with boundaries in order to construct operations on Hochschild cohomology and Hochschild homology. These operations are parametrized by a colored operad involving disks on the surface of a cylinder defined by Kontsevich and Soibleman. The formalism of the proof extends without difficulties to a higher dimensional situation. More precisely, we can replace associative algebras by algebras over the little disks operad of any dimensions, Hochschild homology by factorization (also called topological chiral) homology and Hochschild cohomology by higher Hochschild cohomology. Note that our result works in categories of chain complexes but also in categories of modules over a commutative ring spectrum giving interesting operations on topological Hochschild homology and cohomology.
\end{abstract}

\maketitle

\tableofcontents

Let $A$ be an associative algebra over a field $k$. A famous theorem by Hochschild Kostant and Rosenberg (see \cite{hochschilddifferential}) suggests that the Hochschild homology of $A$ should be interpreted as the graded vector space of differential forms on the non commutative space ``$\on{Spec} A$''. Similarly, the Hochschild cohomology of $A$ should be interpreted as the space of polyvector fields on $\on{Spec} A$.

If $M$ is a smooth manifold, let $\Omega_*(M)$ be the (homologically graded) vector space of de Rham differential forms and $V^*(M)$ be the vector space of polyvector fields (i.e. global sections of the exterior algebra on $TM$). This pair of graded vector spaces supports the following structure:

\begin{itemize}
\item The de Rham differential : $d:\Omega_*(M)\to\Omega_{*-1}(M)$.

\item The cup product of vector fields : $-.-:V^i(M)\otimes V^j(M)\to V^{i+j}(M)$.

\item The Schouten-Nijenhuis bracket : $[-,-]:V^i\otimes V^j\to V^{i+j-1}$.

\item The cap product : $\Omega_i\otimes V^j\to \Omega_{i-j}$ denoted by $\omega\otimes X\mapsto i_X\omega$.

\item The Lie derivative : $\Omega_i\otimes V^j\to\Omega_{i-j+1}$ denoted by $\omega\otimes X\mapsto L_X\omega$.
\end{itemize}

This structure satisfies some properties:

\begin{itemize}
\item The de Rham differential is indeed a differential, i.e. $d\circ d=0$.

\item The cup product and the Schouten-Nijenhuis bracket make $V^*(M)$ into a Gerstenhaber algebra. More precisely, the cup product is graded commutative and the bracket satisfies the Jacobi identity and is a derivation in each variable with respect to the cup product.

\item The cap product and the Lie derivative make $\Omega_*(M)$ into a Gerstenhaber $V^*(M)$-module.

\end{itemize}

The Gerstenhaber module structure means that the following formulas are satisfied
\begin{align*}
L_{[X,Y]}&=[L_X,L_Y]\\
i_{[X,Y]}&=[i_X,L_Y]\\
i_{X.Y}&=i_Xi_Y\\
L_{X.Y}&=L_Xi_Y+(-1)^{|X|}i_XL_Y
\end{align*}
where we denote by $[-,-]$, the (graded) commutator of operators on $\Omega_*(M)$.

Finally we have the following formula called Cartan's formula relating the Lie derivative, the exterior product and the de Rham differential:
\[L_X=[d,i_X]\]

Note that there is even more structure available in this situation. For example, the de Rham differential forms are equipped with a commutative differential graded algebra structure. However we will ignore this additional structure since it is not available in the non commutative case.

There is an operad $\oper{C}alc$ in graded vector spaces such that a $\oper{C}alc$-algebra is a pair $(V^*,\Omega_*)$ together with all the structure we have just mentioned.

It turns out that any associative algebra gives rise to a $\oper{C}alc$-algebra pair:

\begin{theo}
Let $A$ be an associative algebra over a field $k$, let $\HH{}_*(A)$ (resp. $\HC{}^*(A)$) denote the Hochschild homology (resp. cohomology) of $A$, then the pair $(\HC{}^*(A),\HH{}_*(A))$ is an algebra over $\oper{C}alc$.
\end{theo}

It is a natural question to try to lift this action to an action at the level of chains inducing the $\oper{C}alc$-action in homology. This is similar to Deligne conjecture which states that there is an action of the operad of little $2$-disks on Hochschild cochains of an associative algebra inducing the Gerstenhaber structure after taking homology.

Kontsevich and Soibelman in \cite{kontsevichnotes} have constructed a topological colored operad denoted $\oper{KS}$ whose homology is $\oper{C}alc$. The purpose of this paper is to construct an action of $\oper{KS}$ on the pair consisting of topological Hochschild cohomology and topological Hochschild homology. 

More precisely, we prove the following theorem:

\begin{theo}
Let $A$ be an associative algebra in the category of chain complexes over a commutative ring or in the category of modules over a commutative symmetric ring spectrum. Then there is an algebra $(C,H)$ over $\oper{KS}$ such that $C$ is weakly equivalent to the (topological) Hochshchild cohomology of $A$ and $H$ is weakly equivalent to the (topological) Hochschild homology of $A$.
\end{theo}

We also prove a generalization of the above theorem to $\oper{E}_d$-algebras. Hochschild cohomology should be replaced by the derived endomorphisms of $A$ seen as an $\oper{E}_d$-module over itself and Hochschild homology should be replaced by factorization homology (also called chiral homology). We construct obvious higher dimensional analogues of the operad $\oper{KS}$ and show that they describe the action of higher Hochschild cohomology on factorization homology. 

The crucial ingredients in the proof is the swiss-cheese version of Deligne's conjecture (see \cite{thomaskontsevich} or \cite{ginotnotes}) and a study of factorization homology on manifolds with boundaries as defined in \cite{ayalastructured}. 

Note that one could imagine fancier versions of our main theorem using manifolds with corners instead of manifolds with boundaries (the relevant background can be found in \cite{ayalastructured} and \cite{calaquearound}). 

\subsection*{Plan of the paper}

\begin{itemize}
\item The first two sections are just background material about operads and model categories. We have proved the results whenever, we could not find a proper reference, however, this material makes no claim of originality. 

\item The third section is a definition of the little $d$-disk operad and the swiss-cheese operad. Again it is not original and only included to fix notations. 

\item The fourth and fifth sections are devoted to the construction of the operads $\oper{E}_d$ and $\oper{E}_d^\partial$. These are smooth versions of the little $d$ disk operad and the swiss-cheese operad. 

\item We show in the sixth section that $\oper{E}_d$ and $\oper{E}_d^\partial$ are weakly equivalent to the little $d$ disk operad and the swiss-cheese operad. 

\item In the seventh section we construct factorization homology of $\oper{E}_d$ and $\oper{E}_d^\partial$ algebras over a manifold (with boundary in the case of $\oper{E}_d^\partial$) and prove various useful results about it. 

\item In the eighth section, we construct a smooth analogue of the operad $\oper{KS}$ as well as its higher dimensional versions. 

\item Finally in the last section we construct an action of these operads on the pair consisting of higher Hochschild cohomology and factorization homology.
\end{itemize}

\subsection*{Acknowledgements}

This paper is part of the author's Ph.D. thesis at MIT. This work benefited a lot from conversations with Haynes Miller, Clark Barwick, David Ayala, Ricardo Andrade, John Francis and Luis Alexandre Pereira.
\subsection*{Conventions}

In this paper, we denote by $\cat{S}$ the category of simplicial set with its usual model structure. All our categories are implicitely assumed to be enriched in simplicial sets and all our functors are functors of simplicially enriched categories. We use the symbol $\simeq$ to denote a weak equivalence and $\cong$ to denote an isomorphism.

\section{Colored operad}

We recall the definition of a colored operad (also called a multicategory). In this paper we will restrict ourselves to the case of operads in $\S$ but the same definitions could be made in any symmetric monoidal category. Note that we use the word ``operad'' even when the operad has several colors. When we want to specifically talk about operads with only one color, we say ``one-color operad''.

\begin{defi}
An  \emph{operad} in the category of simplicial sets consists of
\begin{itemize}
\item a set of colors $\on{Col}(\oper{M})$

\item for any finite sequence $\{a_i\}_{i\in I}$ in $\on{Col}(\oper{M})$ indexed by a finite set $I$, and any color $b$, a simplicial set:
\[\oper{M}(\{a_i\}_I;b)\]

\item a base point $*\to\oper{M}(a;a)$ for any color $a$

\item for any map of finite sets $f:I\to J$, whose fiber over $j\in J$ is denoted $I_j$, compositions operations
\[\left(\prod_{j\in J}\oper{M}(\{a_i\}_{i\in I_j};b_j)\right)\times\oper{M}(\{b_j\}_{j\in J};c)\to\oper{M}(\{a_i\}_{i\in I};c)\]
\end{itemize}

All these data are required to satisfy unitality and associativity conditions (see for instance \cite[Definition 2.1.1.1]{luriehigher}).

A map of operads $\oper{M}\to\oper{N}$ is a map $f:\on{Col}(\oper{M})\to\on{Col}(\oper{N})$ together with the data of maps
\[\oper{M}(\{a_i\}_I;b)\to\oper{N}(\{f(a_i)\}_I;f(b))\]
compatible with the compositions and units.
\end{defi}

With the above definition, it is not clear that there is a category of operads since there is no set of finite sets. However it is easy to fix this by checking that the only data needed is the value $\oper{M}(\{a_i\}_{i\in I};b)$ on sets $I$ of the form $\{1,\ldots,n\}$. The above definition has the advantage of avoiding unnecessary identification between finite sets.

\begin{rem}\label{symmetric group action}
Note that the last point of the definition can be used with an automorphism $\sigma:I\to I$. Using the unitality and associativity of the composition structure, it is not hard to see that $\oper{M}(\{a_i\}_{i\in I};b)$ supports an action of the group $\on{Aut}(I)$. Other definitions of operads include this action as part of the structure.
\end{rem}

\begin{defi}
Let $\oper{M}$ be an operad. The underlying simplicial category of $\oper{M}$ denoted $\oper{M}^{(1)}$ is the simplicial category whose objects are the colors of $\oper{M}$ and with
\[\Map_{\oper{M}^{(1)}}(m,n)=\oper{M}(\{m\};n)\]
\end{defi}

\begin{nota}
Let $\{a_i\}_{i\in I}$ and $\{b_j\}_{j\in J}$ be two sequences of colors of $\oper{M}$. We denote by $\{a_i\}_{i\in I}\boxplus\{b_j\}_{j\in J}$ the sequence indexed over $I\sqcup J$ whose restriction to $I$ (resp. to $J$) is $\{a_i\}_{i\in I}$ (resp. $\{b_j\}_{j\in J}$).

For instance if we have two colors $a$ and $b$, we can write $a^{\boxplus n}\boxplus b^{\boxplus m}$ to denote the sequence $\{a,\ldots,a,b,\ldots,b\}_{\{1,\ldots,n+m\}}$ with n $a$'s and m $b$'s.
\end{nota}

Any symmetric monoidal category can be seen as an operad:

\begin{defi}
Let $(\cat{A},\otimes,\un_\cat{A})$ be a small symmetric monoidal category enriched in $\S$. Then $\cat{A}$ has an underlying  operad $\oper{U}\cat{A}$ whose colors are the objects of $A$ and whose spaces of operations are given by
\[\oper{U}\cat{A}(\{a_i\}_{i\in I};b)=\Map_{\cat{A}}(\bigotimes_{i\in I}a_i,b)\]
\end{defi}

\begin{defi}
We denote by $\cat{Fin}$ the category whose objects are nonnegative integers $n$ and whose morphisms $n\to m$ are maps of finite sets
\[\{1,\ldots,n\}\to\{1,\ldots,m\}\]

We allow ourselves to write $i\in n$ when we mean $i\in \{1,\ldots,n\}$.
\end{defi}

The construction $\cat{A}\mapsto \oper{U}\cat{A}$ sending a symmetric monoidal category to an operad has a left adjoint that we define now. We will use the boldface letter $\cat{M}$ to denote value of this left adjoint on $\oper{M}$. We will call it the PROP associated to $\oper{M}$.

\begin{defi}
Let $\oper{M}$ be an operad, the objects of the free symmetric monoidal category $\cat{M}$ are given by
\[\on{Ob}(\cat{M})=\bigsqcup_{n\in\on{Ob}(\cat{Fin})}\on{Col}(\cat{M})^n\]

Morphisms are given by
\[\cat{M}(\{a_i\}_{i\in n},\{b_j\}_{j\in m})
=\bigsqcup_{f:n\to m}\prod_{i\in m}\oper{M}(\{a_j\}_{j\in f^{-1}(i)};b_i)\]
\end{defi}

It is easy to check that there is a functor $\cat{M}^2\to\cat{M}$ which on objects is
\[(\{a_i\}_{i\in n},\{b_j\}_{j\in m})\mapsto \{a_1\ldots,a_n,b_1,\ldots,b_m\}\]
A straightforward verification shows that this functor indeed makes $\cat{M}$ into a symmetric monoidal simplicial category.

We denote by $\cat{Fin}_S$ the PROP associated to the initial operad with set of colors $S$. The category $\cat{Fin}_S$ is the category whose objects are pairs $(n,u)$ where $n\in\cat{Fin}$ and $u:n\to S$ is a map. A morphism from $(n,u)$ to $(m,v)$ only exists when $n=m$. In that case, it is the data of an isomorphism $\sigma:n\to n$ which is such that $u=v\circ\sigma$.

\bigskip
Let $\cat{C}$ be a symmetric monoidal simplicial category. For an element $X\in\cat{C}^{S}$ and $x=(n,u)\in\cat{Fin}_S$, we write
\[X^{\otimes x}=\bigotimes_{i\in n}X_{u(i)}\]
Then $x\mapsto X^{\otimes x}$ defines a symmetric monoidal functor $\cat{Fin}_S\to\cat{C}$. 

An $\oper{M}$-algebra in $\cat{C}$ is a map of  operads $\oper{M}\to \oper{U}\cat{C}$. By definition, an algebra over $\oper{M}$ induces a (symmetric monoidal) functor $\cat{M}\to\cat{C}$. We will use the same notation for the two objects and allow oursleves to switch between them without mentioning it. We denote by $\cat{C}[\oper{M}]$ the category of $\oper{M}$-algebras in $\cat{C}$.

Alternatively, we can consider the category $\cat{C}^{\on{Col}(\oper{M})}$ of tuples of elements of $\cat{C}$ indexed by the colors of $\oper{M}$. The operad $\oper{M}$ defines a monad on that category via the formula
\[\oper{M}(X)(c)=\on{colim}_{x\in\cat{Fin}_{\on{Col}(\oper{M})}}\oper{M}(x;c)\otimes X^{\otimes x}\]

The category of $\oper{M}$-algebras in $\cat{C}$ is then the category of algebras over the monad $\oper{M}$.

\subsection*{Right modules over operads}

\begin{defi}
Let $\oper{M}$ be an operad. A \emph{right $\oper{M}$-module} is a simplicial functor
\[R:\cat{M}\op\to \S\]
We denote by $\Mod_{\oper{M}}$ the category of modules over $\oper{M}$.
\end{defi}

\begin{rem}
If $\oper{O}$ is a one-color operad, it is easy to verify that the category of right modules over $\oper{O}$ in the above sense is isomorphic to the category of right modules over $\oper{O}$ in the usual sense (i.e. a right module over the monoid $\oper{O}$ with respect to the monoidal structure on symmetric sequences given by the composition product).
\end{rem}

The category of right modules over $\oper{M}$ has a convolution tensor product. Given $P$ and $Q$ two right modules over $\oper{M}$, we first define their exterior tensor product $P\boxtimes Q$ which is a functor $\cat{M}\op\times\cat{M}\op\to\S$ sending $(m,n)$ to $P(m)\times Q(n)$. The tensor product $P\otimes Q$ is then defined to be the left Kan extension along the tensor product $\mu:\cat{M}\op\times\cat{M}\op\to\cat{M}\op$ of the exterior tensor product $P\boxtimes Q$.

\begin{prop}\label{prop-monoidality of coend}
If $A$ is an $\oper{M}$-algebra, then there is an isomorphism
\[(P\otimes_{\cat{M}}A)\otimes (Q\otimes_{\cat{M}}A)\cong (P\otimes Q)\otimes_{\cat{M}}A\]
\end{prop}

\begin{proof}
By definition, we have
\[(P\otimes Q)\otimes_{\cat{M}}A=\mu_!(P\boxtimes Q)\otimes_{\cat{M}}A\]
By associativity of coends, we have
\[(P\otimes Q)\otimes_{\cat{M}}A\cong P\boxtimes Q\otimes_{\cat{M}\times\cat{M}}\mu^*A\]
Since $A$ is a symmetric monoidal functor, we have an isomorphism $A\boxtimes A\cong \mu^*A$. Thus we have
\[(P\otimes Q)\otimes_{\cat{M}}A\cong P\boxtimes Q\otimes_{\cat{M}\times\cat{M}}A\boxtimes A\]

This last coend is the coequalizer
\[\on{colim}_{m,n,p,q\in\cat{Fin}_{\on{Col}\oper{M}}}[P(m)\times Q(n)\times\cat{M}(p,m)\times\cat{M}(q,n)]\otimes A(p)\otimes A(q)\]
\[\rightrightarrows\on{colim}_{m,n\in\cat{Fin}_{\on{Col}\oper{M}}}[P(m)\times Q(n)]\otimes A(m)\otimes A(n)\]

Each factor of the tensor product $(P\otimes_{\cat{M}}A)\otimes(Q\otimes_{\cat{M}}A)$ can be written as a similar coequalizer. 

Each of these coequalizers is a reflexive coequalizer. Since the tensor product $\cat{C}\times\cat{C}\to\cat{C}$ preserves reflexive coequalizer in both variables separately, according to \cite[Proposition 1.2.1]{fressemodules}, it sends reflexive coequalizers in $\cat{C}\times\cat{C}$ to reflexive coequalizers in $\cat{C}$ . The proposition follows immediately from this fact.
\end{proof}

\subsection{Operadic pushforward}

Let $\oper{M}$ be an operad with set of colors $S$, we have a symmetric monoidal functor $i:\cat{Fin}_S\to \cat{M}$, where $\cat{Fin}_S$ is the PROP associated to $\oper{I}$ the initial colored operad with set of colors $S$. This induce a forgetful functor
\[i^*:\Mod_{\oper{M}}\to\Mod_{\oper{I}}\]

\begin{prop}
The functor $i^*$ is symmetric monoidal.
\end{prop}

\begin{proof}
First, it is obvious that this functor is lax monoidal.

Since colimits in $\Mod_{\oper{M}}$ and $\Mod_{\oper{I}}$ are computed objectwise, the functor $i^*$ commutes with colimits. By the universal property of the Day convolution product (see for instance \cite[Proposition 2.1.]{isaacsonsymmetric}), $i^*$ is symmetric monoidal if and only if its restriction to representables is symmetric monoidal. 

Thus, let $x$ and $y$ be two objects of $\cat{M}$, we want to prove that the canonical map
\[\cat{M}(i(-),x)\otimes\cat{M}(i(-),y)\to\cat{M}(i(-),x\boxplus y)\]
is an isomorphism.

By definition, we have
\[\cat{M}(i(z),x\boxplus y)=\bigsqcup_{f:z\to x\boxplus y}\prod_{i\in x\boxplus y}\oper{M}(f^{-1}(i);i)\]

A map $z\to x\boxplus y$ in $\cat{Fin}_S$ is entirely determined by a choice of splitting $z\cong u\boxplus v$ and the data of a map $u\to x$ and a map $v\to y$. Thus we have
\begin{align*}
\cat{M}(i(z),x\boxplus y)&\cong\bigsqcup_{z\cong u\boxplus v}\bigsqcup_{f:u\to x,g:v\to y}\prod_{i\in x, j\in y}\oper{M}(f^{-1}(i);i)\times\oper{M}(g^{-1}(j);j)\\
                         &\cong\bigsqcup_{z\cong u\boxplus v}\cat{M}(i(u),x)\times\cat{M}(i(v),y)\\
                         &\cong(\cat{M}(i(-),x)\otimes\cat{M}(i(-),y))(z)
\end{align*}
\end{proof}

Let us consider more generally a map $u:\oper{M}\to\oper{N}$ between colored operad. It induces a functor
\[u^*:\Mod_{\oper{N}}\to\Mod_{\oper{M}}\]

\begin{prop}\label{prop-monoidality of pullback}
The functor $u^*$ is symmetric monoidal.
\end{prop}

\begin{proof}
Let $S$ be the set of colors of $\oper{M}$ and $T$ be the set of colors of $\oper{N}$. We have a commutative diagram of operads
\[
\xymatrix{
\oper{I}_S\ar[r]^i\ar[d]_v& \oper{M}\ar[d]^u\\
\oper{I}_T\ar[r]_j& \oper{N}
}
\]
where $\oper{I}_S$ (resp. $\oper{I}_T$) is the initial operads with set of colors $S$ (resp. $T$). The functor $u^*$ and $v^*$ are obviously lax monoidal. Since $i^*$ and $j^*$ are conservative and symmetric monoidal by the previous proposition, it suffices to prove that $v^*$ is symmetric monoidal. 

Let $X$ and $Y$ be two objects of $\Mod_{\oper{I}_T}$, we want to prove that the map
\[X(v-)\otimes Y(v-)\to X\otimes Y(v-)\]
is an isomorphism. Let $p\in \on{Ob}(\cat{Fin}_S)$, the value of the left hand side at $p$ can be written as
\[\on{colim}_{(q,r)\in(\cat{Fin}_S\times\cat{Fin}_S)_{/p}} X(vq)\times Y(vr)\]

On the other hand, the value of the right hand side at $p$ can be written as
\[\on{colim}_{(x,y)\in(\cat{Fin}_T\times\cat{Fin}_T)_{/vp}} X(x)\times Y(y)\]

The map $v$ induces a functor $(\cat{Fin}_S\times\cat{Fin}_S)_{/p}\to (\cat{Fin}_T\times\cat{Fin}_T)_{/vp}$ that is easily checked to be an equivalence of categories. This concludes the proof.
\end{proof}

\begin{coro}
Assume that $\cat{C}$ is cocomplete. Let $\alpha:\oper{M}\to\oper{N}$ be a morphism of operads. Then, the left Kan extension functor
\[\alpha_!:\on{Fun}(\cat{M},\cat{C})\to \on{Fun}(\cat{N},\cat{C})\]
restricts to a functor
\[\alpha_!:\cat{C}[\oper{M}]\to\cat{C}[\oper{N}]\]
\end{coro}

\begin{proof}
According to proposition \ref{prop-monoidality of coend}, it suffices for the functor from $\cat{N}$ to $\Mod_{\oper{M}}$ sending $n$ to $\cat{N}(\alpha(-),n)$ to be symmetric monoidal functor. This is precisely implied by proposition \ref{prop-monoidality of pullback} and the fact that the Yoenda's embedding is symmetric monoidal.
\end{proof}

\begin{defi}
We keep the notations of the previous proposition. The $\oper{N}$-algebra $\alpha_!(A)$ is called the \emph{operadic left Kan extension} of $A$ along $\alpha$.
\end{defi}

\begin{prop}\label{prop-operadic vs categorical left Kan extension}
If $\alpha:\oper{M}\to\oper{N}$ is a map between colored operads, then the forgetful functor $\alpha^*:\cat{C}[\oper{N}]\to\cat{C}[\oper{M}]$ is right adjoint to the functor $\alpha_!$.
\end{prop}

\begin{proof}
First, we observe that $\alpha^*:\cat{C}[\oper{N}]\to\cat{C}[\oper{M}]$ is the restriction of $\alpha^*:\on{Fun}(\cat{N},\cat{C})\to\on{Fun}(\cat{M},\cat{C})$ which explains the apparent conflict of notations.

Let $S$ (resp. $T$) be the set of colors of $\oper{M}$ (resp. $\oper{N}$). Let $\oper{I}_S$ and $\oper{I}_T$ be the initial object in the category of operads with set of colors $S$ (resp. $T$). We define $\oper{M}'=\oper{M}\sqcup^{\oper{I}_S}\oper{I}_T$. The map $\alpha$ can be factored as the obvious map $\oper{M}\to\oper{M}'$ followed by the map $\oper{M}'\to\oper{N}$ which induces the identity map on colors. It suffices to prove the proposition for each of these two maps. The case of the first map is trivial, thus we can assume that $\oper{M}\to\oper{N}$ is the identity map on colors.

The forgetful functor $\cat{C}[\oper{M}]\to\on{Fun}(\cat{M},\cat{C})$ preserves reflexive coequalizers and similarly for $\cat{C}[\oper{N}]\to\on{Fun}(\cat{N},\cat{C})$. On the other hand, any $\oper{M}$-algebra $A$ can be expressed as the following reflexive coequalizer
\[\oper{M}\oper{M}A\rightrightarrows \oper{M}A\to A\]
where the top map $\oper{M}\oper{M}A\to \oper{M}A$ is induced by the monad structure on $\oper{M}$ and the second map is induced by the algebra structure $\oper{M}A\to A$.

Let us write temporarily $L$ for the left adjoint of $\alpha^*:\cat{C}[\oper{N}]\to\cat{C}[\oper{M}]$. According to the previous paragraph, it suffices to prove the proposition for $A=\oper{M}X$ a free $\oper{M}$-algebra on $X\in\cat{C}^{\on{Col}\oper{M}}$. In that case, $LA=\oper{N}X$. On the other hand for $c\in\on{Col}(\oper{M})$, we have
\[\alpha_!A(c)=\oper{N}(\alpha-,c)\otimes_{\cat{M}}A\]

A trivial computation shows that $A=\oper{M}X$ is the left Kan extension of $X^{\otimes-}$ along the obvious map $\beta:\cat{Fin}_{S}\to\cat{M}$ and similarly $\oper{N}X$ is the left Kan extension of $X^{\otimes-}$ along $\alpha\circ\beta$. Thus, we have
\[\alpha_!A\cong\alpha_!\beta_!X\cong(\alpha\circ\beta)_!X=\oper{N}X\]
\end{proof}

\section{Homotopy theory of operads and modules}

In this section we collect a few facts about the homotopy theory in categories of algebras in a reasonable symmetric monoidal simplicial model category.

\begin{defi}
Let $\oper{M}$ be an operad with set of colors $S$. A right module $X:\cat{M}\op\to\S$ is said to be $\Sigma$-cofibrant if its restriction along the map $\cat{Fin}_S\to\S$ is a projectively cofibrant object of $\on{Fun}(\cat{Fin}_S\op,\S)$. 

An operad $\oper{M}$ is said to be $\Sigma$-cofibrant if for each $m\in\on{Col}(\oper{M})$, the right module $\oper{M}(-;m)$ is $\Sigma$-cofibrant over $\oper{M}$.
\end{defi}

\begin{rem}
Note that $\cat{Fin}_S$ is a groupoid. Thus a functor $X$ in $\on{Fun}(\cat{Fin}_S\op,\S)$ is projectively cofibrant if and only if $X(c)$ is an $\on{Aut}(c)$-cofibrant space for each $c$ in $\cat{Fin}_S$. This happens in particular, if $\on{Aut}(c)$ acts freely on $X(c)$.

In particular, if $\oper{O}$ is a single-color operad, it is $\Sigma$-cofibrant if and only if for each $n$, $\oper{O}(n)$ is cofibrant as a $\Sigma_n$-space.
\end{rem}

\begin{defi}
A \emph{weak equivalence} between operads is a morphism of operad $f:\oper{M}\to\oper{N}$ which is a bijection on objects and such that for each $\{m_i\}_{i\in I}$ a finite set of colors of $\oper{M}$ and each $m$ a color of $\oper{M}$, the map
\[\oper{M}(\{m_i\};m)\to\oper{N}(\{f(m_i)\};f(m))\]
is a weak equivalence.
\end{defi}

\begin{rem}
This is not the most general form of weak equivalences of operads but we will not need a fancier definition in this paper.
\end{rem}

\subsection*{Algebras in categories of modules over a ring spectrum}

If $E$ is a commutative monoid in the category $\Spec$ of symmetric spectra, we define $\Mod_E$ to be the category of right modules over $E$ equipped with the positive model structure (see \cite{schwedeuntitled}). This category is a closed symmetric monoidal left proper simplicial model category. There is another model structure $\Mod_E^a$ on the same category with the same weak equivalences but more cofibrations. In particular, the unit $E$ is cofibrant in $\Mod_E^a$ but not in $\Mod_E$. The model category $\Mod_E^a$ is also a symmetric monoidal left proper simplicial model category.

\begin{theo}\label{theo-operads in ring spectra}
Let $E$ be a commutative symmetric ring spectrum. Then the positive model structure on $\Mod_E$ is such that for any operad $\oper{M}$, the category $\Mod_E[\oper{M}]$ has a model structure in which the weak equivalences and fibrations are colorwise. Moreover if $A$ is a cofibrant algebra over an operad $\oper{M}$, then $A$ is cofibrant for the absolute model structure.
\end{theo}

\begin{proof}
This is done in \cite{pavlovsymmetric}. 
\end{proof}

Moreover, this model structure is homotopy invariant:

\begin{theo}\label{theo-Quillen equivalence}
Let $\alpha:\oper{M}\to\oper{N}$ be a weak equivalence of operads. Then the adjunction
\[\alpha_!:\Mod_E[\oper{M}]\leftrightarrows\Mod_E[\oper{N}]:\alpha^*\]
is a Quillen equivalence.
\end{theo}

\begin{proof}
This is also done in \cite{pavlovsymmetric}.
\end{proof}

If $R$ is a commutative $\mathbb{Q}$-algebra, the category $\cat{Ch_*}(R)$ of chain complexes over $R$ with its projective model structure satisfies a similar theorem. Note that the category $\cat{Ch}_*(R)$ is not simplicial. Nevertheless, the functor $C_*$ which assigns to a simplicial set its normalized $R$-chain complex is lax monoidal. Therefore, it makes sense to speak about an algebra over a simplicial operad $\oper{M}$, this is just an algebra over the operad $C_*(\oper{M})$. In that case the cofibrant algebras are colorwise cofibrant. Proofs can be found in \cite{hinichrectification}.

This result remains true for symmetric spectra in more general modal categories. A case of great interest is the case of motivic spectra. That is symmetric spectra with respect to $\mathbf{P}^1_k$ in the category of based simplicial presheaves over the site of smooth schemes over a field $k$. More details about this can be found in \cite{pavlovsymmetric}.

\subsection*{Berger-Moerdijk model structure}

\begin{theo}
Let $\cat{C}$ be a left proper simplicial symmetric monoidal cofibrantly generated model category. Assume that $\cat{C}$ has a monoidal fibrant replacement functor and a cofibrant unit. Then all $\Sigma$-cofibrant operads are admissible in $\cat{C}$. Moreover, if $A$ is a cofibrant algebra over a $\Sigma$-cofibrant operad $\oper{M}$, then $A$ is colorwise cofibrant in $\cat{C}$.
\end{theo}

\begin{proof} The proof of the last axiom is done in \cite[Theorem 4.1.]{bergerresolution}. The second claim is proved in \cite{fressemodules} in the case of single-color operads. Unfortunately, we do not know a reference in the case of colored operads.
\end{proof}

For instance $\S$ and $\cat{Top}$ satisfy the conditions of the theorem. Every object is fibrant in $\cat{Top}$ and the functor $X\mapsto \on{Sing}(|X|)$ is a symmetric monoidal fibrant replacement functor in $\S$. If $R$ is a commutative ring, the category $s\Mod_R$ of simplicial $R$-modules satisfies the conditions. 

If $\cat{T}$ is a small site, the category of simplicial sheaves over $\cat{T}$ with its injective model structure (in which cofibrations are monomorphisms and weak equivalences are local weak equivalences) satisfies the conditions of the theorem.

\subsection*{Homotopy invariance of operadic coend}

From now on, we let $(\cat{C},\otimes,\un)$ be the category $\Mod_E$ with its positive model structure. We write $\cat{C}$ instead of $\Mod_E$ to emphasize that the argument work in greater generality modulo some small modifications. In particular, the results we give extend to the model category of chain complexes over $R$ a $\mathbb{Q}$-algebra. They also extend to a category that satisfies the Berger-Moerdijk assumptions if one restricts to $\Sigma$-cofibrant operads and modules.

We want to study the homotopy invariance of coends of the form $P\otimes_{\cat{M}}A$ for $A$ an $\oper{M}$-algebra and $P$ a right module over $\oper{M}$.

\begin{prop}\label{invariance of operadic coend}
Let $\oper{M}$ be an operad and let $\cat{M}$ be the PROP associated to $\oper{M}$. Let $A:\cat{M}\to\cat{C}$ be an algebra. Then
\begin{enumerate}
\item Let $P:\cat{M}\op\to \S$ be a right module. Then $P\otimes_{\cat{M}}-$ preserves weak equivalences between cofibrant $\oper{M}$-algebras.
\item If $A$ is a cofibrant algebra, the functor $-\otimes_{\cat{M}}A$ is a left Quillen functor from right modules over $\oper{M}$ to $\cat{C}$ with the absolute model structure.
\item Moreover the functor $-\otimes_{\cat{M}}A$ preserves all weak equivalences between right modules.
\end{enumerate}
\end{prop}

\begin{proof}
For $P$ a functor $\cat{M}\op\to\S$, we denote by $\oper{M}_{P}$ the  operad whose colors are $\on{Col}(\oper{M})\sqcup\infty$ and whose operations are as follows:
\begin{align*}
\oper{M}_{P}(\{m_1,\ldots,m_k\},n)&=\oper{M}(\{m_1,\ldots,m_k\},n)\;\;\textrm{if }\infty\notin\{m_1,\ldots,m_k,n\}\\
\oper{M}_{P}(\{m_1,\ldots,m_k\};\infty)&=P(\{m_1,\ldots,m_k\})\;\;\textrm{if }\infty\notin\{m_1,\ldots,m_k\}\\
\oper{M}_{P}(\{\infty\};\infty)&=*\\
\oper{M}_{P}(\{m_1,\ldots,m_k\},n)&=\varnothing\;\;\textrm{in any other case}
\end{align*}

There is an obvious operad map $\alpha_P:\oper{M}\to\oper{M}_P$. Moreover by \ref{prop-operadic vs categorical left Kan extension} we have
\[\on{ev}_{\infty}(\alpha_P)_!A\cong\oper{M}_P(-,\infty)\otimes_{\cat{M}}A\cong P\otimes_{\cat{M}}A\]
where $\on{ev}_\infty$ denotes the functor that evaluate an $\oper{M}_P$-algebra at the color $\infty$.

\emph{Proof of the first claim}. If $A\to B$ is a weak equivalence between cofibrant $\oper{M}$-algebras, then $(\alpha_P)_!A$ is weakly equivalent to $(\alpha_P)_!B$ since $(\alpha_P)_!$ is a left Quillen functor. To conclude the proof, we observe that the functor $\on{ev}_{\infty}$ preserves all weak equivalences.

\emph{Proof of the second claim}. In order to show that $P\mapsto P\otimes_{\cat{M}}A$ is left Quillen it suffices to check that it sends generating (trivial) cofibrations to (trivial) cofibrations. 

For $m\in\on{Ob}(\cat{M})$, denote by $\iota_m$ the functor $\S\to\on{Fun}(\on{Ob}(\cat{M}),\S)$ sending $X$ to the functor sending $m$ to $X$ and everything else to $\varnothing$. Denote by $F_{\cat{M}}$ the left Kan extension functor
\[F_{\cat{M}}:\on{Fun}(\on{Ob}(\cat{M})\op,\S)\to\on{Fun}(\cat{M}\op,\cat{C})\]

We can take as generating (trivial) cofibrations the maps of the form $F_{\cat{M}} \iota_m I$ ($F_{\cat{M}} \iota_m J$) for $I$ (resp. $J$), the generating cofibrations (resp. trivial cofibrations) of $\S$. We have:
\[F_{\cat{M}}\iota_m I\otimes_{\cat{M}}A\cong I\otimes A(m)\]
Since $A$ is cofibrant as an algebra its value at each object of $\cat{M}$ is cofibrant in the absolute model structure. Since the absolute model structure is a simplicial model category we are done.

\emph{Proof of the third claim}. Let $P\to Q$ be a weak equivalence between functors $\cat{M}\op\to \S$. This induces a weak equivalence between  operads $\beta:\oper{M}_P\to\oper{M}_Q$. It is clear that $\alpha_Q=\beta\circ\alpha_P$, therefore $(\alpha_Q)_!A=\beta_!(\alpha_P)_!A$. We apply $\beta^*$ to both side and get
\[\beta^*\beta_!(\alpha_P)_!A=\beta^*(\alpha_Q)_!A\]

Since $(\alpha_P)_!A$ is cofibrant and $\beta^*$ preserves all weak equivalences, the adjunction map \[(\alpha_P)_!A\to\beta^*\beta_!(\alpha_P)_!A\]
is a weak equivalence by \ref{theo-Quillen equivalence}. Therefore the obvious map
\[(\alpha_P)_!A\to\beta^*(\alpha_Q)_!A\]
is a weak equivalence. 

If we evaluate this at the color $\infty$, we find a weak equivalence
\[P\otimes_{\cat{M}}A\to Q\otimes_{\cat{M}}A\]
\end{proof}

\subsection*{Operadic vs categorical homotopy left Kan extension}

As we have seen in \ref{prop-operadic vs categorical left Kan extension}, given a map of operad $\alpha:\oper{M}\to\oper{N}$, the operadic left Kan extension $\alpha_!$ applied to an algebra $A$ over $\oper{M}$ coincides with the left Kan extension of the functor $A:\cat{M}\to\cat{C}$. We call the latter the categorical left Kan extension of $A$. 

It is not clear that the derived functors of these two different left Kan extension coincide. Indeed, in the case of the derived operadic left Kan extension, we take a cofibrant replacement of the $\oper{M}$-algebra $A$ as an algebra and in the case of the categorical left Kan extension we take a cofibrant replacement of the functor $A:\cat{M}\to\cat{C}$ in the category of functors with the projective model structure. However, it turns out that the two constructions coincide.

\begin{prop}\label{Operadic vs categorical}
Let $\alpha:\oper{M}\to\oper{N}$ be a morphism of operads. Let $A$ be an algebra over $\oper{M}$. The derived operadic left Kan extension $\L\alpha_!(A)$ is weakly equivalent to the homotopy left Kan extension of $A:\cat{M}\to\cat{C}$ along the induced map $\cat{M}\to\cat{N}$.
\end{prop}

\begin{proof}
Let $QA\to A$ be a cofibrant replacement of $A$ as an $\oper{M}$-algebra. We can consider the bar construction of the functor $QA:\cat{M}\to\cat{C}$:
\[\on{B}_\bullet(\cat{N}(\alpha-,n),\cat{M},QA)\]

We know that $QA$ is objectwise cofibrant in the absolute model structure. Therefore, the bar construction is Reedy-cofibrant in the absolute model structure and computes the categorical left Kan extension of $A$.

We can rewrite this simplicial object as
\[\on{B}_\bullet(\cat{N}(\alpha-,n),\cat{M},\cat{M})\otimes_{ \cat{M}}QA\]

The geometric realization is
\[|\on{B}_\bullet(\cat{N}(\alpha-,n),\cat{M},\cat{M})|\otimes_{ \cat{M}}QA\]

It is a classical fact that the map
\[|\on{B}_\bullet(\cat{N}(\alpha-,n),\cat{M},\cat{M})|\to\cat{N}(\alpha-,n)\]
is a weak equivalence of functors on $\cat{M}$. Therefore by \ref{invariance of operadic coend}, the Bar construction is weakly equivalent to $\alpha_!QA$ which is exactly the derived operadic left Kan extension of $A$.
\end{proof}

\section{The little $d$-disk operad}

In this section, we give a traditional definition of the \emph{little $d$-disk operad} $\oper{D}_d$ as well as a definition of the \emph{swiss-cheese operad} $\oper{SC}_d$ which we denote $\oper{D}_d^\partial$. The swiss-cheese operad, originally defined by Voronov (see \cite{voronovswiss} for a definition when $d=2$ and \cite{thomaskontsevich} for a definition in all dimensions), is a variant of the little $d$-disk operad which describes the action of an $\oper{D}_d$-algebra on an $\oper{D}_{d-1}$-algebra.

\subsection*{Space of rectilinear embeddings}
Let $D$ denote the open disk of dimension $d$, $D=\{x\in\mathbb{R}^d, \|x\|<1\}$.

\begin{defi}
Let $U$ and $V$ be connected subsets of $\mathbb{R}^d$, let $i_U$ and $i_V$ denote the inclusion into $\mathbb{R}$. We say that $f:U\to V$ is a \emph{rectilinear embedding} if there is an element $L$ in the subgroup of $\on{Aut}(\mathbb{R}^d)$ generated by translation and dilations with positive factor such that
\[i_V\circ f=L\circ i_U\]
\end{defi}

We extend this definition to disjoint unions of open subsets of $\mathbb{R}^d$:

\begin{defi}
Let $U_1, \ldots, U_n$ and $V_1,\ldots,V_m$ be finite families of connected subsets of $\mathbb{R}^d$. The notation $U_1\sqcup\ldots\sqcup U_n$ denotes the coproduct of $U_1,\ldots U_n$ in the category of topological spaces. We say that a map from $U_1\sqcup\ldots\sqcup U_n$ to $V_1\sqcup\ldots\sqcup V_m$ is a \emph{rectilinear embedding} if it satisfies the following properties:
\begin{enumerate}

\item Its restriction to each component can be factored as $U_i\to V_j\to V_1\sqcup\ldots\sqcup V_m$ where the second map is the obvious inclusion and the first map is a rectilinear embedding $U_i\to V_j$.

\item The underlying map of sets is injective.
\end{enumerate}

We denote by $\Emb_{lin}(U_1\sqcup\ldots\sqcup U_n, V_1\sqcup\ldots\sqcup V_m)$ the subspace of $\Map(U_1\sqcup\ldots\sqcup U_n, V_1\sqcup\ldots\sqcup V_m)$ whose points are rectilinear embeddings.
\end{defi}

Observe that rectilinear embeddings are stable under composition.

\subsection*{The $d$-disk operad}
\begin{defi}
The \emph{linear $d$-disk operad}, denoted $\oper{D}_d$, is the operad in topological spaces whose $n$-th space is $\Emb_{lin}(D^{\sqcup n},D)$ with the composition induced from the composition of rectilinear embeddings.
\end{defi}

There are variants of this definition but they are all equivalent to this one. In the above definition $\oper{D}_d$ is an operad in topological spaces. By applying the functor $\on{Sing}$, we get an operad in $\S$. We use the same notation for the topological and the simplicial operad.

\subsection*{The swiss-cheese operad}

As before, we denote by $D$, the $d$-dimensional disk and by $H$ the $d$-dimensional half-disk
\[H=\{x=(x_1,\ldots,x_d\}),\|x\|<1, x_{d}\geq 0\}\]

\begin{defi}
The \emph{linear $d$-dimensional swiss-cheese operad}, denoted $\oper{D}_{d}^\partial$, has two colors $z$ and $h$ and its mapping spaces are
\begin{align*}
\oper{D}_{d}^\partial(z^{\boxplus n},z)&=\Emb_{lin}(D^{\sqcup n},D)\\
\oper{D}_{d}^\partial(z^{\boxplus n}\boxplus h^{\boxplus m},h)&=\Emb_{lin}^\partial(D^{\sqcup n}\sqcup H^{\sqcup m},H)
\end{align*}
where the $\partial$ superscript means that we restrict to embeddings preserving the boundary.
\end{defi}

\begin{prop}
The full suboperad of $\oper{D}_{d}^\partial$ on the color $z$ is isomorphic to $\oper{D}_{d}$ and the full suboperad on the color $h$ is isomorphic to $\oper{D}_{d-1}$.
\end{prop}

\begin{proof}
Easy.
\end{proof}

\begin{prop}
The evaluation at the center of the disks induces weak equivalences
\begin{align*}
\oper{D}_d(n)&\goto{\simeq}\on{Conf}(n,D)\\
\oper{D}_d^\partial(z^{\boxplus n}\boxplus h^{\boxplus m},h)&\goto{\simeq}\on{Conf}(m,\partial H)\times\on{Conf}(n,H-\partial H)
\end{align*}

\end{prop}

\begin{proof}
These maps are Hurewicz fibration whose fibers are contractible.
\end{proof}

\section{Homotopy pullback in $\cat{Top}_W$}

The material of this section can be found in \cite{andrademanifolds}. We have included it mainly for the reader's convenience and also to give a proof of \ref{hopbW vs hopb} which is mentioned without proof in \cite{andrademanifolds}.

\subsection*{Homotopy pullback in $\cat{Top}$}
Let us start by recalling the following well-known proposition:

\begin{prop}
Let
\[\xymatrix{
& X\ar[d]^f\\
Y\ar[r]_g&Z}\]
be a diagram in $\cat{Top}$. The homotopy pullback of that diagram can be constructed as the space of triples $(x,p,y)$ where $x$ is a point in $X$, $y$ is a point in $Y$ and $p$ is a path from $f(x)$ to $g(y)$ in $Z$.\hfill$\square$
\end{prop}

\subsection*{Homotopy pullback in $\cat{Top}_W$}
Let $W$ be a topological space. There is a model structure on $\cat{Top}_W$ the category of topological spaces over $W$ in which cofibrations, fibrations and weak equivalences are reflected by the forgetful functor $\cat{Top}_W\to\cat{Top}$. We want to study homotopy pullbacks in $\cat{Top}_W$

We denote a space over $W$ by a single capital letter like $X$ and we write $p_X$ for the structure map $X\to W$.

Let $I=[0,1]$, for $Y$ an object of $\cat{Top}_W$, we denote by $Y^I$ the cotensor in the category $\cat{Top}_W$. Concretely, $Y^I$ is the space of paths in $Y$ whose image in $W$ is a constant path.

\begin{defi}
Let $f:X\to Y$ be a map in $\cat{Top}_W$. We denote by $Nf$ the following pullback in $\cat{Top}_W$:
\[\xymatrix{
Nf\ar[r]\ar[d]& Y^I\ar[d]\\
X\ar[r]_f&Y}\]
\end{defi}
Concretely, $Nf$ is the space of pairs $(x,p)$ where $x$ is a point in $X$ and $p$ is  a path in $Y$ whose value at $0$ is $f(x)$ and lying over a constant path in $W$. 

We denote by $p_f$, the map $Nf\to Y$ sending a path to its value at $1$.

\begin{prop}
Let
\[\xymatrix{
& X\ar[d]^f\\
Y\ar[r]&Z}\]
be a diagram in $\cat{Top}_W$ in which $X$ and $Z$ are fibrant (i.e. the structure maps $p_X$ and $p_Z$ are fibrations) then the pullback of the following diagram in $\cat{Top}_W$ is a model for the homotopy pullback:
\[\xymatrix{
& Nf\ar[d]_{p_f}\\
Y\ar[r]&Z}\]
\end{prop}

Concretely, this proposition is saying that the homotopy pullback is the space of triple $(x,p,y)$ where $x$ is a point in $X$, $y$ is a point in $Y$ and $p$ is a path in $Z$ between $f(x)$ and $g(y)$ lying over a constant path in $W$.

\begin{proof}[Proof of the proposition]
The proof is similar to the analogous result in $\cat{Top}$, it suffices to check that the map $p_f:Nf\to Z$ is a fibration in $\cat{Top}_W$ which is weakly equivalent to $X\to Z$. Since the category $\cat{Top}_W$ is right proper, a pullback along a fibration is always a homotopy pullback.
\end{proof}

From now on when we talk about a homotopy pullback in the category $\cat{Top}_W$, we mean the above specific model.  

\begin{rem}\label{fibration}
The map from the homotopy pullback to $Y$ is a fibration. If $X$, $Y$, $Z$ are fibrants, the homotopy pullback can be computed in two different ways but they are clearly isomorphic. In particular, the map from the homotopy pullback to $X$ is also a fibration.
\end{rem}

\subsection*{Comparison of homotopy pullbacks in $\cat{Top}$ and in $\cat{Top}_W$}

For a diagram
\[\xymatrix{
& X\ar[d]^f\\
Y\ar[r]&Z}\]
in $\cat{Top}$ (resp. $\cat{Top}_W$), we denote by $\on{hpb}(X\rightarrow Z\leftarrow Y)$ (resp. $\on{hpb}_W(X\rightarrow Z\leftarrow Y)$) the above model  of homotopy pullback in $\cat{Top}$ (resp. $\cat{Top}_W$).

Note that there is an obvious inclusion
\[\on{hpb}_W(X\rightarrow X\leftarrow Y)\to\on{hpb}(X\rightarrow Z\leftarrow Y)\]
which sends a path (which happens to be constant in $W$) to itself.

\begin{prop}\label{hopbW vs hopb}
Let $W$ be a topological space and $X\rightarrow Y\leftarrow Z$ be a diagram in $\cat{Top}_W$ in which the structure maps $Z\to W$ and $Y\to W$ are fibrations, then the inclusion
\[\on{hpb}_W(X\rightarrow Y\leftarrow Z)\to\on{hpb}(X\rightarrow Y\leftarrow Z)\]
is a weak equivalence.
\end{prop}

\begin{proof}\footnote{The following proof is due to Ricardo Andrade}
Let us consider the following commutative diagram
\[\xymatrix{
\on{hopb}_W(X\rightarrow Y\leftarrow Z)\ar[d]\ar[r]&\on{hopb}(X\rightarrow Y\leftarrow Z)\ar[d]\ar[r]& X\ar[d]\\
\on{hopb}_W(Y\rightarrow Y\leftarrow Z)\ar[d]\ar[r]&\on{hopb}(Y\rightarrow Y\leftarrow Z)\ar[r]\ar[d]&Y\\
W\ar[r]&W^I& \\
}\]
The map $\on{hopb}(Y\rightarrow Y\leftarrow Z)\to W^I$ sends a triple $(y,p,z)$ to the image of the path $p$ in $W$. The map $W\to W^I$ sends a point in $W$ to the constant map at that point. All other maps should be clear. 

It is straightforward to check that each square is cartesian.

The category $\cat{Top}_W$ is right proper. This implies that a pullback along a fibration is always a homotopy pullback.

Now we make the following three observations: 

(1) The map $\on{hopb}(Y\rightarrow Y\leftarrow Z)\to W^I$ is a fibration. Indeed it can be identified with the obvious map $Y^I\times_Y Z\to W^I\times_W W$ and $Y^I\to W^I$, $Z\to W$ and $Y\to W$ are fibrations. This implies that the bottom square is homotopy cartesian.

(2) The middle row of the diagram $\on{hopb}_W(Y\rightarrow Y\leftarrow Z)\to Y$ is a fibration because of remark \ref{fibration}. A priori it is a fibration in $\cat{Top}_W$ but this is equivalent to being a fibration in $\cat{Top}$. This implies that the big horizontal rectangle is homotopy cartesian. 

(3) The map $\on{hopb}(Y\rightarrow Y\leftarrow Z)\to Y$ is a fibration for the same argument we used in observation (2). This implies that the right-hand side square is homotopy cartesian.

If we combine (2) and (3) we find that the top left-hand side square is homotopy cartesian. If we combine that with (1), we find that the big horizontal rectangle is homotopy cartesian. The map $W\to W^I$ is a weak equivalence. Therefore the map
\[\on{hopb}_W(X\rightarrow Y\leftarrow Z)\to\on{hopb}(X\rightarrow Y\leftarrow Z)\]
is a weak equivalence as well.
\end{proof}

\section{Embeddings between structured manifolds}

This section again owes a lot to \cite{andrademanifolds}. In particular, the definition \ref{framed embedding} can be found in that reference. We then make an analogous definitions of embedding spaces for framed manifolds with boundary.

\subsection*{Topological space of embeddings}
There is a topological category whose objects are $d$-manifolds possibly with boundary and mapping object between $M$ and $N$ is $\Emb(M,N)$, the topological space of smooth embeddings with the weak $C^1$ topology. The reader should look at \cite{hirschdifferential} for a definition of this topology. We want to emphasize that this topology is metrizable, in particular $\Emb(M,N)$ is paracompact.

\begin{rem}
If one is only interested in the homotopy type of this topological space. One could work with the $C^r$-topology for any $r$ (even $r=\infty$) instead of the $C^1$-topology. The choice of taking the weak (as opposed to strong topology) however is a serious one. The two topologies coincide when the domain is compact. However the strong topology does not have continuous composition maps
\[\Emb(M,N)\times\Emb(N,P)\to\Emb(M,P)\]
when $M$ is not compact.
\end{rem}

\subsection*{Embeddings between framed manifolds}

For a manifold $M$ possibly with boundary, we denote by $\on{Fr}(TM)\to M$ the principal $\on{GL}(d)$-bundle of frames of the tangent bundle of $M$.

\begin{defi}
A \emph{framed $d$-manifold} is a pair $(M,\sigma_M)$ where $M$ is a $d$-manifold and $\sigma_M$ is a smooth section of the principal $\on{GL}(d)$-bundle $\on{Fr}(TM)$.
\end{defi}

If $M$ and $N$ are two framed $d$-manifolds, we define a space of framed embeddings denoted by $\Emb_f(M,N)$ as in \cite{andrademanifolds}: 

\begin{defi}\label{framed embedding}
Let $M$ and $N$ be two framed $d$-dimensional manifolds. The \emph{topological space of framed embeddings from $M$ to $N$}, denoted $\Emb_f(M,N)$, is given by the following homotopy pullback in the category of topological spaces over $\Map(M,N)$:
\[\xymatrix
{\Emb_f(M,N)\ar[r]\ar[d] & \Map(M,N)\ar[d]\\
\Emb(M,N)\ar[r] & \Map_{\on{GL}(d)}(\on{Fr}(TM),\on{Fr}(TN))}
\]

The right hand side map is obtained as the composition
\[\Map(M,N)\to\Map_{\on{GL}(d)}(M\times\on{GL}(d),N\times\on{GL}(d))\cong\Map_{\on{GL}(d)}(\on{Fr}(TM),\on{Fr}(TN))\]  
where the first map is obtained by taking the product with $\on{GL}(d)$ and the second map is induced by the identification $\on{Fr}(TM)\cong M\times\on{GL}(d)$ and $\on{Fr}(TN)\cong N\times \on{GL}(d)$.
\end{defi}

It is not hard to show that there are well defined composition maps
\[\Emb_f(M,N)\times\Emb_f(N,P)\to\Emb_f(M,P)\]
allowing the construction of a topological category $f\Mfld_d$ (see \cite{andrademanifolds}).

\begin{rem}
Taking a homotopy pullback in the category of spaces over $\Map(M,N)$ is not strictly necessary. Taking the homotopy pullback of the underlying diagram of spaces would have given the same homotopy type by \ref{hopbW vs hopb}. However, this definition has the psychological advantage that any point in the space $\Emb_f(M,N)$ lies over a point in $\Map(M,N)$ in a canonical way. If we had taken the homotopy pullback in the category of spaces, the resulting object would have had two distinct maps to $\Map(M,N)$, one given by the upper horizontal arrow and the other given as the composition $\Emb_f(M,N)\to\Emb(M,N)\to\Map(M,N)$.
\end{rem}

\subsection*{Embeddings between framed manifolds with boundary}

If $N$ is a manifold with boundary, $n$ a point of the boundary, and $v$ is a vector in $TN_n-T(\partial N)_n$, we say that $v$ is pointing inward if it can be represented as the tangent vector at $0$ of a curve $\gamma:[0,1)\to N$ with $\gamma(0)=n$. 

\begin{defi}
A \emph{$d$-manifold with boundary} is a pair $(N,\phi)$ where $N$ is a $d$-manifold with boundary in the traditional sense and $\phi$ is an isomorphism of $d$-dimensional vector bundles over $\partial N$
\[\phi:T(\partial N)\oplus\mathbb{R}\to TN_{|\partial N}\]
which is required to restrict to the canonical inclusion $T(\partial N)\to TN_{|\partial N}$, and which is such that for any $n$ on the boundary, the point $1\in\mathbb{R}$ is sent to an inward pointing vector through the composition
\[\mathbb{R}\to T_n(\partial N)\oplus\mathbb{R}\goto{\phi_n}T_nN\]
\end{defi}

In other words, our manifolds with boundary are equipped with smooth family of inward pointing vector at each point of the boundary. We require maps between manifolds with boundary to preserve the direction defined by these vectors:

\begin{defi}
Let $(M,\phi)$ and $(N,\psi)$ be two $d$-manifolds with boundary, we define the space $\Emb(M,N)$ to be the topological space of smooth embeddings from $M$ into $N$ sending $\partial M$ to $\partial N$, preserving the splitting of the tangent bundles along the boundary $T(\partial M)\oplus\mathbb{R}\to T(\partial N)\oplus\mathbb{R}$. The topology on this space is the weak $C^1$-topology.
\end{defi}

In particular, if $\partial M$ is empty, $\Emb(M,N)=\Emb(M,N-\partial N)$. If $\partial N$ is empty and $\partial M$ is not empty, $\Emb(M,N)=\varnothing$.

\bigskip
We now introduce framings on manifolds with boundary. We require a framing to interact well with the boundary.

\begin{defi}
Let $(N,\phi)$ be a $d$-manifold with boundary. We say that a section $\sigma_N$ of $\on{Fr}(TN)$ is \emph{compatible with the boundary} if for each point $n$ on the boundary of $N$ there is a splitting-preserving isomorphism
\[T_n(\partial N)\oplus\mathbb{R}\goto{\phi_n}T_nN\goto{\sigma_N}\mathbb{R}^{d-1}\oplus\mathbb{R}\]
whose restriction to the $\mathbb{R}$-summand is multiplication by a positive real number.

A framed $d$-manifold with boundary is a $d$-manifold with boundary together with the datum of a compatible framing.
\end{defi}

\begin{defi}
Let $M$ and $N$ be two framed $d$-manifolds with boundary. We denote by $\Map^\partial_{\on{GL}(d)}(\on{Fr}(TM),\on{Fr}(TN))$ the topological space of $\on{GL}(d)$-equivariant maps sending $\on{Fr}(TM_{|\partial M})$ to $\on{Fr}(TN_{|\partial N})$ and preserving the framings that are compatible with the boundary.
\end{defi}

\begin{defi}\label{framed embedding with boundary}
Let $M$ and $N$ be two framed $d$-manifolds with boundary. The \emph{topological space of framed embeddings} from $M$ to $N$, denoted $\Emb_f(M,N)$, is the following homotopy pullback in the category of topological spaces over $\Map((M,\partial M),(N,\partial N))$
\[\xymatrix
{\Emb_f(M,N)\ar[r]\ar[d] & \Map((M,\partial M),(N,\partial N))\ar[d]\\
\Emb(M,N)\ar[r] & \Map^\partial_{\on{GL}(d)}(\on{Fr}(TM),\on{Fr}(TN))}
\] 
\end{defi}

Concretely, a point in $\Emb_f(M,N)$ is a pair $(\phi,p)$ where $\phi:M\to N$ is an embedding of manifolds with boundary and $p$ is the data at each point $m$ of $M$ of a path between the two trivializations of $T_mM$ (the one given by the framing of $M$ and the one given by pulling back the framing of $N$ along $\phi$). These paths are required to vary smoothly with $m$. Moreover if $m$ is  a point on the boundary, the path between the two trivializations of $T_mM$ must be such that at any time, the first $d-1$-vectors are in $T_m\partial M\subset T_mM$ and the last vector is a positive multiple of the inward pointing vector which is part of our definition of a manifold with boundary.

\section{Homotopy type of spaces of embeddings}

We want to analyse the homotopy type of spaces of embeddings described in the previous section. None of the result presented here are surprising. Some of them are proved in greater generality in \cite{cerftopologie}. However the author of \cite{cerftopologie} is working with the strong topology on spaces of embeddings and for our purposes, we needed to use the weak topology.

As usual, $D$ denotes the $d$-dimensional open disk of radius $1$ and $H$ is the upper half-disk of radius $1$

We will make use of the following two lemmas.

\begin{lemm}\label{filtered}
Let $X$ be a topological space with an increasing filtration by open subsets $X=\bigcup_{n\in\mathbb{N}}U_n$. Let $Y$ be another space and $f:X\to Y$ be a continuous map such that for all $n$, the restriction of $f$ to $U_n$ is a weak equivalence. Then $f$ is a weak equivalence.
\end{lemm}

\begin{proof}
We can apply theorem \ref{lurie}. This theorem implies that $X$ is equivalent to the homotopy colimit of the open sets $U_n$ which immediately yields the desired result.
\end{proof}

\begin{lemm}{(Cerf)}\label{locally trivial fibration}
Let $G$ be a topological group and let $p:E\to B$ be a map of $G$-topological spaces. Assume that for any $x\in B$, there is a neighborhood of $x$ on which there is a section of the map
\begin{align*}
G&\to B\\
g&\mapsto g.x
\end{align*}
Then if we forget the action, the map $p$ is a locally trivial fibration. In particular, if $B$ is paracompact, it is a Hurewicz fibration.
\end{lemm}

\begin{proof}
See \cite{cerftheoremes}.
\end{proof}

Let $\Emb^*(D,D)$  (resp. $\Emb^{\partial,*}(H,H)$) be the topological space of self embeddings of $D$ (resp. $H$) mapping $0$ to $0$.

\begin{prop}
The ``derivative at the origin'' map
\[\Emb^*(D,D)\to\on{GL}(d)\]
is a Hurewicz fibration and a weak equivalence. The analogous result for the map
\[\Emb^*(H,H)\to\on{GL}(d-1)\]
also holds.
\end{prop}

\begin{proof}
Let us first show that the derivative map
\[\Emb^*(D,D)\to\on{GL}(d)\]
is a Hurewicz fibration. 

The group $\on{GL}(d)$ acts on the source and the target and the derivative map commutes with this action. We use lemma \ref{locally trivial fibration}, it suffices to show that for any $u\in \on{GL}(d)$, we can define a section of the multiplication by $u$ map
\[\on{GL}(d)\to\on{GL}(d)\]
which is trivial.

Now we show that the fibers are contractible. Let $u\in\on{GL}(d)$ and let $\Emb^u(D,D)$ be the space of embedding whose derivative at $0$ is $u$, we want to prove that $\Emb^u(D,D)$ is contractible. It is equivalent but more convenient to work with $\mathbb{R}^d$ instead of $D$. Let us consider the following homotopy:

\begin{align*}
\Emb^u(\mathbb{R}^d,\mathbb{R}^d)\times(0,1]&\to\Emb^u(\mathbb{R}^d,\mathbb{R}^d)\\
(f,t)&\mapsto \left(x\mapsto \frac{f(tx)}{t}\right)
\end{align*}

At $t=1$ this is the identity of $\Emb^u(D,D)$. We can extend this homotopy by declaring that its value at $0$ is constant with value the linear map $u$. Therefore, the inclusion $\{u\}\to\Emb^u(D,D)$ is a deformation retract.

The proof for $H$ is similar.
\end{proof}

\begin{prop}
Let $M$ be a manifold (possibly with boundary). The map
\[\Emb(D,M)\to \on{Fr}(TM)\]
is a weak equivalence and a Hurewicz fibrations. Similarly the map
\[\Emb(H,M)\to\on{Fr}(T\partial M)\]
is a weak equivalence and a Hurewicz fibration.
\end{prop}

\begin{proof}
The fact that these maps are Hurewicz fibrations will follow again from lemma \ref{locally trivial fibration}. We will assume that $M$ has a framing because this will make the proof easier and we will only apply this result with framed manifolds. However the result remains true in general.

Let us do the proof for $D$. The derivative map
\[\Emb(D,M)\to\on{Fr}(TM)\cong M\times\on{GL}(d)\]
is equivariant with respect to the action of the group $\on{Diff}(M)\times\on{GL}(d)$. It suffices to show that for any $x\in \on{Fr}(TM)$, the ``action on $x$'' map
\[\on{Diff}(M)\times\on{GL}(d)\to M\times\on{GL}(d)\]
has a section in a neighborhood of $x$. Clearly it is enough to show that for any $x$ in $M$, the ``action on $x$'' map
\[\on{Diff}(M)\to M\]
has a section in a neighborhood of $x$

We can restrict to neighborhoods $U$ such that $U\subset \bar{U}\subset V\subset M$ in which $U$ and $V$ are diffeomorphic to $\mathbb{R}^d$.

Let us consider the group $\on{Diff}^c(V)$ of diffeomorphisms of $V$ that are the identity outside a compact subset of $V$. Clearly we can prolong one of these diffeomorphism by the identity and there is a well define inclusion of topological groups
\[\on{Diff}^c(V)\to\on{Diff}(M)\]
Now we have made the situation local. It is equivalent to construct a map
\[\phi:D\to\on{Diff}^c(\mathbb{R}^d)\]
with the property that $\phi(x)(0)=x$.

Let $f$ be a smooth function from $\mathbb{R}^d$ to $\mathbb{R}$ which is such that
\begin{itemize}
\item $f(0)=1$
\item $\|\nabla f\|\leq \frac{1}{2}$
\item $f$ is compactly supported
\end{itemize}

We claim that 
\[\phi(x)(u)=f(u)x+u\] 
satisfies the requirement which proves that 
\[\Emb(D,M)\to\on{Fr}(TM)\]
is a Hurewicz fibration. The case of $H$ is similar.

Now let us prove that this derivative maps are weak equivalences. 

We have the following commutative diagram
\[\xymatrix
{\Emb(D,M)\ar[r]\ar[d] & \on{Fr}(TM)\ar[d]\\
M\ar[r]^= & M}
\]

Both vertical maps are Hurewicz fibration, therefore it suffices to check that the induced map on fibers is a weak equivalence. We denote by $\Emb^m(D,M)$ the subspace consisting of those embeddings sending $0$ to $m$. Hence all we have to do is prove that for any point $m\in M$ the derivative map $\Emb^m(D,M)\to\on{Fr}T_mM$ is a weak equivalence. If $M$ is $D$ and $m=0$, this is the previous proposition. In general, we pick an embedding $f:D\to M$ centered at $m$. Let $U_n\subset \Emb^m(D,M)$ be the subspace of embeddings mapping $D_n$ to the image of $f$ (where $D_n\subset D$ is the subspace of points of norm at most $1/n$). Clearly $U_n$ is open in $\Emb^m(D,M)$ and $\bigcup _n U_n=\Emb^m(D,M)$, by \ref{filtered} it suffices to show that the map $U_n\to\on{Fr}(T_mM)$ is a weak equivalence for all $n$.

Clearly the inclusion $U_1\to U_n$ is a deformation retract for all $n$, therefore, it suffices to check that $U_1\to\on{Fr}(T_mM)$ is a weak equivalence. Equivalently, it suffices to prove that $\Emb^0(D,D)\to\on{GL}(d)$ is a weak equivalence and this is exactly the previous proposition.
\end{proof}

This result extends to disjoint union of copies of $H$ and $D$ with a similar proof.

\begin{prop}
The derivative map
\[\Emb(D^{\sqcup p}\sqcup H^{\sqcup q},M)\to\on{Fr}(T\on{Conf}(p,M-\partial M))\times \on{Fr}(T\on{Conf}(q,\partial M))\]
is a weak equivalence and a Hurewicz fibration.
\end{prop}

In the case of framed embeddings, we have the following result:

\begin{prop}\label{framed embeddings of disks}
The evaluation at the center of the disks induces a weak equivalence
\[\Emb_f(D^{\sqcup p}\sqcup H^{\sqcup q},M)\to\on{Conf}(p,M-\partial M)\times \on{Conf}(q,\partial M)\]
\end{prop}

\begin{proof}
To simplify notations, we restrict to studying $\Emb_f(H,M)$, the general case is similar. By definition \ref{framed embedding with boundary} and proposition \ref{hopbW vs hopb}, we need to study the following homotopy pullback:
\[\xymatrix
{ & \Map((H,\partial H),(M,\partial M))\ar[d]\\
\Emb(H,M)\ar[r] & \Map^\partial_{\on{GL}(d-1)}(\on{Fr}(TH),\on{Fr}(TM))}
\]

This diagram is weakly equivalent to
\[\xymatrix
{ & \partial M \ar[d]\\
\on{Fr}(T(\partial M))\ar[r] & \on{Fr}(T(\partial M))}
\]
where the bottom map is the identity. Therefore, $\Emb_f(H,M)\simeq \partial M$.
\end{proof}

\begin{prop}
Let $M$ be a $d$-manifold with compact boundary and let $S$ be a compact $(d-1)$-manifold without boundary. The ``restriction to the boundary'' map
\[\Emb(S\times[0,1),M)\to \Emb(S,\partial M)\]
is a Hurewicz fibration and a weak equivalence.
\end{prop}

\begin{proof}
Note that an embedding between compact connected manifold without boundary is necessarily a diffeomorphism. Therefore the two spaces in the proposition are empty unless $S$ is diffeomorphic to a disjoint union of connected components of $\partial M$. 

Let us assume that $S$ and $\partial M$ are connected and diffeomorphic. The general case follows easily from this particular case. 

We first prove that this map is a Hurewicz fibration. We use the criterion \ref{locally trivial fibration}. The map
\[\Emb(S\times[0,1),M)\to \Emb(S,\partial M)\]
is obviously equivariant with respect to the obvious right action of $\on{Diff}(S)$ on both sides. Therefore, for any $f\in\Emb(S,\partial M)$, we need to define a section of the ``action on $f$'' map
\[\on{Diff}(S)\to\Emb(S,\partial M)\]
but this map is by hypothesis a diffeomorphism.

Now let us prove that each fiber is contractible. Let $\alpha$ be a diffeomorphism $S\to \partial M$. We need to prove that the space $\Emb^\alpha(S\times[0,1),M)$ consisting of embeddings whose restriction to the boundary is $\alpha$ is contractible.

Let us choose one of these embeddings $\phi:S\times[0,1)\to M$ and let's denote its image by $C$. For $n>0$, let $U_n$ be the subset of $\Emb^\alpha(S\times[0,1),M)$ consisting of embeddings $f$ with the property that $f(S\times[0,\frac{1}{n}])\subset C$. By definition of the weak $C^1$-topology, $U_n$ is open in $\Emb^\alpha(S\times[0,1),M)$, moreover $\Emb^\alpha(S\times[0,1),M)=\bigcup_n U_n$, therefore by \ref{filtered}, it is enough to prove that $U_n$ is contractible for all $n$.

Let us consider the following homotopy:
\[
H:\left[0,1-\frac{1}{n}\right]\times U_n\to U_n
\]
\[
(t,f)\mapsto ((s,u)\mapsto f(s,(1-t)u))
\]

It is a homotopy between the identity of $U_n$ and the inclusion $U_1\subset U_n$. Therefore $U_1$ is a deformation retract of each of the $U_n$ and all we have to prove is that $U_1$ is contractible. But each element of $U_1$ factors through $C=\on{Im}\phi$, hence all we need to do is prove the lemma when $M=S\times[0,1)$ and $\alpha=\id$. It is equivalent and notationally simpler to do it for $S\times\mathbb{R}_{\geq 0}$\footnote{The following was suggested to us by S{\o}ren Galatius}.

For $t\in(0,1]$, let $h_t:S\times\mathbb{R}_{\geq 0}\to S\times\mathbb{R}_{\geq 0}$ be the diffeomorphism sending $(s,u)$ to $(s,tu)$

Let us consider the following homotopy
\[
(0,1]\times\Emb^\id(S\times\mathbb{R}_{\geq 0},S\times\mathbb{R}_{\geq 0})\to\Emb^\id(S\times\mathbb{R}_{\geq 0},S\times\mathbb{R}_{\geq 0})
\]
\[
(t,f)\mapsto h_{1/t}\circ f\circ h_t\]

At time $1$, this is the identity of $\Emb^\id(S\times[0,+\infty),S\times[0,+\infty))$. At time $0$ it has as limit the map
\[(s,u)\mapsto \left(s,u\frac{\partial f}{\partial u}(s,0)\right)\]
that lies in the subspace of $\Emb^\id(S\times[0,+\infty),S\times[0,+\infty))$ consisting of element which are of the form $(s,u)\mapsto(s,a(s)u)$ for some smooth function $a:S\to\mathbb{R}_{>0}$. This space is obviously contractible and we have shown that it is deformation retract of $\Emb^\id(S\times[0,+\infty),S\times[0,+\infty))$.
\end{proof}

\begin{prop}\label{restriction to the boundary}
Let $M$ be a framed $d$-manifold with compact boundary. The ``restriction to the boundary'' map
\[\Emb_f(S\times[0,1),M)\to \Emb_f(S,\partial M)\]
is a weak equivalence. 
\end{prop}

\begin{proof}
There is a restriction map comparing the pullback diagram defining $\Emb_f(S\times[0,1),M)$ to the pullback diagram defining $\Emb_f(S,\partial M)$. Each of the three maps is a weak equivalence (one of them because of the previous proposition) therefore, the homotopy pullbacks are equivalent.
\end{proof}

We are now ready to define the operads $\oper{E}_d$, $\oper{E}_d^{\partial}$.

\begin{defi}
The operad $\oper{E}_d$ of \emph{little $d$-disks} is the simplicial operad whose $n$-th space is $\Emb_f(D^{\sqcup n},D)$.
\end{defi}

Note that there is an inclusion of operads
\[\oper{D}_d\to\oper{E}_d\]

\begin{prop}\label{framed vs linear}
This map is a weak equivalence of operads.
\end{prop}
\begin{proof}
It is enough to check it degreewise. The map
\[\oper{D}_d\to\on{Conf}(n,D)\]
is a weak equivalence which factors through $\oper{E}_d(n)$ by \ref{framed embeddings of disks}, the map $\oper{E}_d(n)\to\on{Conf}(n,D)$ is a weak equivalence.
\end{proof}

\begin{defi}
The  operad $\oper{E}_d^{\partial}$ is a colored operad with two colors $z$ and $h$ and with
\begin{align*}
\oper{E}_d^\partial(z^{\boxplus n};z)&=\oper{E}_d(n)\\
\oper{E}_d^\partial(z^{\boxplus n}\boxplus h^{\boxplus m};h)&=\Emb_f(D^{\sqcup n}\sqcup H^{\sqcup m},H)
\end{align*}
\end{defi}

\begin{prop}
The obvious inclusion of operads
\[\oper{D}_d^\partial\to\oper{E}_d^\partial\]
is a weak equivalence of operads.
\end{prop}

\begin{proof}
Similar to \ref{framed vs linear}.
\end{proof}

\section{Factorization homology}

In this section, we define factorization homology of $\oper{E}_d$-algebras and $\oper{E}_d^\partial$-algebras. The paper \cite{ayalastructured} defines factorization homology of manifolds with various kind of singularities. The only originality of this section is the language of model categories as opposed to $\infty$-categories.

Let $\mathfrak{M}$ be the set of framed $d$ manifolds whose underlying manifold is a submanifold of $\mathbb{R}^{\infty}$. Note that $\mathfrak{M}$ contains at least one element of each diffeomorphism class of framed $d$-manifold.

\begin{defi}
We denote by $f\oMfld_d$ an operad whose set of colors is $\mathfrak{M}$ and with mapping objects:
\[f\oMfld_d(\{M_1,\ldots,M_n\},M)=\Emb_f(M_1\sqcup\ldots\sqcup M_n,M)\]

As usual, we denote by $f\Mfld_d$ the free symmetric monoidal category on the operad $f\oMfld_d$.

We can see $D\subset\mathbb{R}^d\subset\mathbb{R}^{\infty}$ as an element of $\mathfrak{M}$. The operad $\oper{E}_d$ is the full suboperad of $f\oMfld_d$ on the color $D$.  The category $\cat{E}_d$ is the full subcategory of $f\Mfld_d$ on objects of the form $D^{\sqcup n}$ with $n$ a nonnegative integer.
\end{defi}

Similarly, we define $\mathfrak{M}^\partial$ to be the set of submanifold of $\mathbb{R}^{\infty}$ possibly with boundary. $\mathfrak{M}^\partial$ contains at least one element of each diffeomorphism class of framed $d$-manifold with boundary.

\begin{defi}
We denote by $f\oMfld^{\partial}_d$ the operad whose set of colors is $\mathfrak{M}^\partial$ and with mapping objects:
\[f\oMfld^{\partial}_d(\{M_1,\ldots,M_n\},M)=\Emb_f^\partial(M_1\sqcup\ldots\sqcup M_n,M)\]

We denote by $f\Mfld_d^\partial$ the free symmetric monoidal category on the operad $f\oMfld^{\partial}_d$. 

The suboperad $\oper{E}_d^\partial$ is the full suboperad of $f\oMfld_d^\partial$ on the colors $D$ and $H$.
\end{defi}

From now on, we assume that $\cat{C}$ is a cofibrantly generated symmetric monoidal simplicial model category with a good theory of algebras over $\Sigma$-cofibrant operads.

\begin{defi}\label{factorization}
Let $A$ be an object of $\cat{C}[\oper{E}_d]$. We define \emph{factorization homology with coefficients in $A$} to be the derived operadic left Kan extension of $A$ along the map of operads $\oper{E}_d\to f\oMfld_d$. 

We denote by $\int_MA$ the value at the manifold $M$ of factorization homology. By definition, $M\mapsto \int_M A$ is a symmetric monoidal functor.
\end{defi}

We have $\int_MA=\Emb_f(-,M)\otimes_{\cat{E}_d}QA$ where $QA\to A$ is a cofibrant replacement in the category $\cat{C}[\oper{E}_d]$. We use the fact that the operad $\oper{E}_d$ is $\Sigma$-cofibrant and that the right module $\Emb_f(-,M)$ is $\Sigma$-cofibrant.

We can define factorization homology of an object of $ f\Mfld_d^\partial$ with coefficients in an algebra over $\oper{E}_d^\partial$.

\begin{defi}
Let $(B,A)$ be an algebra over $\oper{E}_d^\partial$ in $\cat{C}$. \emph{Factorization homology with coefficients in $(B,A)$} is the derived operadic left Kan extension of $(B,A)$ along the obvious inclusion of operads $\oper{E}_d^\partial\to f\oMfld_d^\partial$. We write $\int_M(B,A)$ to denote the value at $M\in f\Mfld_d^\partial$ of the induced functor.
\end{defi}

Again, we have $\int_M(B,A)=\Emb_f^\partial(-,M)\otimes_{\cat{E}_d^\partial}Q(B,A)$ where $Q(B,A)\to(B,A)$ is a cofibrant replacement in the category $\cat{C}[\oper{E}_d^\partial]$. We use the fact that $\oper{E}_d^\partial$ is $\Sigma$-cofibrant and that $\Emb_f^{\partial}(-,M)$ is $\Sigma$-cofibrant as a right module over $\oper{E}_d^\partial$.

\subsection*{Factorization homology as a homotopy colimit}

In this section, we show that factorization homology can be expressed as the homotopy colimit of a certain functor on the poset of open sets of $M$ that are diffeomorphic to a disjoint union of disks. Note that this result in the case of manifolds without boundary is proved in \cite{luriehigher}. We assume that $\cat{C}$ is a symmetric monoidal simplicial cofibrantly generated model category with a good theory of algebras over $\Sigma$-cofibrant operads and satisfying proposition \ref{Operadic vs categorical}. As we have shown, proposition \ref{Operadic vs categorical} is satisfied if $\cat{C}$ has a cofibrant unit or if $\cat{C}$ is $L_Z\Mod_E$.

We will rely heavily on the following theorem:

\begin{theo}\label{lurie}
Let $X$ be a topological space and $\cat{U}(X)$ be the poset of open subsets of $X$. Let $\chi:\cat{A}\to\cat{U}(X)$ be a functor from a small discrete category $\cat{A}$. For a point $x\in X$, denote by $\cat{A}_x$ the full subcategory of $A$ whose objects are those that are mapped by $\chi$ to open sets containing $x$. Assume that for all $x$, the nerve of $\cat{A}_x$ is contractible. Then the obvious map:
\[\on{hocolim}\chi\to X\]
is a weak equivalence. 
\end{theo}

\begin{proof}
See \cite{luriehigher} Theorem A.3.1. p. 971.
\end{proof}

Let $M$ be an object of $ f\Mfld_d$. Let $\cat{D}(M)$ the poset of subset of $M$ that are diffeomorphic to a disjoint union of disks. Let us choose for each object $V$ of $\cat{D}(M)$ a framed diffeomorphism $V\cong D^{\sqcup n}$ for some uniquely determined $n$. Each inclusion $V\subset V'$ in $\cat{D}(M)$ induces a morphism $D^{\sqcup n}\to D^{\sqcup n'}$ in $\cat{E}_d$ by composing with the chosen parametrization. Therefore each choice of parametrization induces a functor $\cat{D}(M)\to\cat{E}_d$. Up to homotopy this choice is unique since the space of automorphisms of $D$ in $ \cat{E}_d$ is contractible. 

In the following we assume that we have one of these functors $\delta:\cat{D}(M)\to\cat{E}_d$. We fix a cofibrant algebra $A: \cat{E}_d\to\cat{C}$.

\begin{lemm}\label{hocolim of functors}
The obvious map:
\[\on{hocolim}_{V\in \cat{D}(M)}\Emb_f(-,V)\to\Emb_f(-,M)\]
is a weak equivalence in $\on{Fun}( \cat{E}_d,\S)$.
\end{lemm}

\begin{proof}
It suffices to prove that for each $n$, there is a weak equivalence in spaces:
\[\on{hocolim}_{V\in \cat{D}(M)}\Emb_f(D^{\sqcup n},V)\simeq\Emb_f(D^{\sqcup n},M)\]

We can apply theorem \ref{lurie} to the functor:
\[\cat{D}(M)\to\cat{U}(\Emb_f(D^{\sqcup n},M))\]
sending $V$ to $\Emb_f(D^{\sqcup n},V)\subset\Emb_f(D^{\sqcup n},M)$. For a given point $\phi$ in $\Emb_f(D^{\sqcup n},M)$, we have to show that the poset of open sets $V\in \cat{D}(M)$ such that $\on{im}(\phi)\subset V$ is contractible. But this poset is filtered, thus its nerve is contractible. 
\end{proof}

\begin{coro}\label{left}
We have:
\[\int_M A\simeq\on{hocolim}_{V\in\cat{D}(M)}\int_{\delta(V)} A\]
\end{coro}

\begin{proof}
\sloppy 
By \ref{Operadic vs categorical}, we know that $\int_M A$ is weakly equivalent to the Bar construction $\on{B}(\Emb_f(-,M), \cat{E}_d,A)$. Therefore we have:
\[\int_MA\simeq \on{B}(*,\cat{D}(M),\on{B}(\Emb_f(-,-),\cat{E}_d,A))\]
The right hand side is the realization of a bisimplicial object and its value is independant of the order in which we do the realization.
\end{proof}

\begin{coro}\label{nb}
There is a weak equivalence:
\[\int_MA\simeq\on{hocolim}_{V\in\cat{D}(M)} A(\delta (V))\]
\end{coro}

\begin{proof}
By \ref{left} the left-hand side is weakly equivalent to:
\[\on{hocolim}_{V\in\cat{D}(M)}\int_{\delta(V)}A\]
Let $U$ be an object of $\cat{E}_d$. The object $\int_U A$ is the coend :
\[\Emb_f(-,U)\otimes_{\cat{E}_d}A\]

Yoneda's lemma implies that this coend is isomorphic to $A(U)$. Moreover, this isomorphism is functorial in $U$. Therefore we have the desired identity.
\end{proof}

We want to use a similar approach for manifolds with boundaries. Let $M$ be an object of $f\Mfld_{d-1}$ and let $M\times[0,1)$ be the object of $ f\Mfld_d^\partial$ whose framing is the direct sum of the framing of $M$ and the obvious framing of $[0,1)$. We identify $\cat{D}(M)$ with the poset of open sets of $M\times[0,1)$ of the form $V\times[0,1)$ with $V$ an open set of $M$ that is diffeomorphic to a disjoint union of disks. As before we can pick a functor $\delta:\cat{D}(M)\to\cat{E}_d^\partial$.

\begin{lemm}
The obvious map: 
\[\on{hocolim}_{V\in\cat{D}(M)}\Emb_f(-,V\times[0,1))\to\Emb_f(-,M\times[0,1))\] 
is a weak equivalence in $\on{Fun}(( \cat{E}_d^\partial)\op,\S)$.
\end{lemm}

\begin{proof}
It suffices to prove that for each $p,q$, there is a weak equivalence in spaces:
\[\on{hocolim}_{V\in \cat{D}(M)}\Emb_f(D^{\sqcup p}\sqcup H^{\sqcup q},V\times[0,1))\simeq\Emb_f(D^{\sqcup p}\sqcup H^{\sqcup q},M\times[0,1))\]

It suffices to show, by \ref{lurie}, that for any $\phi\in\Emb(D^{\sqcup p}\sqcup H^{\sqcup q},M\times[0,1))$, the poset $\cat{D}(M)_\phi$ (which is the subposet of $\cat{D}(M)$ on open sets $V$ that are such that $V\times[0,1)\subset M\times[0,1)$ contains the image of $\phi$) is contractible. But it is easy to see that $\cat{D}(M)_\phi$ is filtered. Thus it is contractible.
\end{proof}

\begin{prop}
Let $(B,A): \cat{E}_d^\partial\to\cat{C}$ be a cofibrant $\oper{E}_d^\partial$-algebra, then we have:
\[\int_{M\times[0,1)}(B,A)\simeq \on{hocolim}_{V\in\cat{D}(M)}(B,A)(\delta(V))\]
\end{prop}

\begin{proof}
The proof is a straightforward modification of \ref{nb}.
\end{proof}

There is a morphism of operad $\oper{E}_{d-1}\to\oper{E}_d^\partial$ sending the unique color of $\oper{E}_{d-1}$ to $H$. Indeed $H$ is diffeomorphic to the product of the $(d-1)$-dimensional disk with $[0,1)$. Hence, for $(B,A)$ an algebra over $\oper{E}_d^\partial$, $A$ has an induced $\oper{E}_{d-1}$-strucure.

\begin{prop}\label{coro}
Let $(B,A)$ be an $\oper{E}_d^\partial$-algebra, then we have a weak equivalence:
\[\int_{M\times[0,1)}(B,A)\simeq \int_MA\]
\end{prop}

\begin{proof}
Let $\delta':\cat{D}(M)\to \cat{E}_{d-1}$ be defined as before. Then $\delta$ can be take to be the composite of $\delta'$ and the map $\cat{E}_{d-1}\to\cat{E}_d^{\partial}$.  

Now we prove the proposition. Because of the previous proposition, the left hand side is weakly equivalent to
$\on{hocolim}_{V\in\cat{D}(M)}(B,A)(\delta (V))$. But $(B,A)(\delta(V))$ is $A(\delta'(V))$. Therefore, by \ref{nb} $\on{hocolim}_{V\in\cat{D}(M)}(B,A)(\delta (V))$ is weakly equivalent to $\int_M A$
\end{proof}

\section{$\oper{KS}$ and its higher versions.}

In this section, we recall the definition of the operad $\oper{KS}$ defined in \cite{kontsevichnotes}. We construct an equivalent version of that operad as well as higher dimensional analogues of it.

\begin{defi}
Let $D$ be the $2$-dimensional disk. An injective continuous map $D\to S^1\times(0,1)$ is said to be \emph{rectilinear} if it can be factored as
\[D\goto{l}\mathbb{R}\times(0,1)\to\mathbb{R}\times(0,1)/\mathbb{Z}=S^1\times(0,1)\]
where the map $l$ is rectilinear and the second map is the quotient by the $\mathbb{Z}$-action.

We say that an embedding $S^1\times [0,1)\to S^1\times[0,1)$ is rectilinear if it is of the form $(z,t)\mapsto (z+z_0,at)$ for some fixed $z_0\in S^1$ and $a\in (0,1]$.

We denote by $\Emb_{lin}^{\partial}(S^1\times[0,1)\sqcup D^{\sqcup n},S^1\times[0,1)$ the topological space of injective maps whose restriction to each disk and to $S^1\times[0,1)$ is rectilinear.
\end{defi}

\begin{defi}
The Kontsevich-Soibelman's operad $\oper{KS}$ has two colors $a$ and $m$ and its spaces of operations are as follows
\begin{align*}
\oper{KS}(a^{\boxplus n};a)&=\oper{D}_2(n)\\
\oper{KS}(a^{\boxplus n}\boxplus m;m)&=\Emb_{lin}^\partial(S^1\times[0,1)\sqcup D^{\sqcup n},S^1\times[0,1))
\end{align*}
Any other space of operations is empty.
\end{defi}

Now we define generalizations of $\oper{KS}$. 

\begin{defi}
Let $S$ be a $(d-1)$-manifold with framing $\tau$. We define $S_\tau^{\circlearrowright}\oMod$ to be the operad with two colors $a$ and $m$ and spaces of operations are as follows
\begin{align*}
S_\tau^{\circlearrowright}\oMod(a^{\boxplus n};a)&=\oper{E}_d(n)\\
S_\tau^{\circlearrowright}\oMod(a^{\boxplus n}\boxplus m;m)&=\Emb_f^\partial(S\times[0,1)\sqcup D^{\sqcup n},S\times[0,1))
\end{align*}

\end{defi}

The category $S_\tau^{\circlearrowright}\Mod$ is the category whose objects are disjoint unions of copies of $S\times[0,1)$ and $D$. 

\begin{prop}\label{restriction fibration}
Let $S$ be a compact connected $(d-1)$-manifold. Let $N$ be a manifold with a boundary diffeomorphic to $S$ and let $M$ be an object of $S_\tau^{\circlearrowright}\Mod$ which can be expressed as a disjoint union
\[M=P\sqcup Q\]
in which one of the first factor is of the form $S\times[0,1)\sqcup D^{\sqcup n}$ and the other is a disjoint union of disks. Then the restriction maps
\[\Emb_f(M,N)\to\Emb_f(P,N)\]
is a fibration.
\end{prop}

\begin{proof}
The category $S_\tau^{\circlearrowright}\Mod$ is a symmetric monoidal category. One can consider the category $\on{Fun}(S_\tau^{\circlearrowright}\Mod\op,\S)$. It is a symmetric monoidal category for the convolution tensor product. The Yoneda's embedding:
\[S_\tau^{\circlearrowright}\Mod\to\on{Fun}(S_\tau^{\circlearrowright}\Mod\op,\S)\]
is a symmetric monoidal functor. By the enriched Yoneda's lemma, the space $\Emb_f(M,N)$ can be identified with the space of natural transformations
\[\Map_{\on{Fun}(S_\tau^{\circlearrowright}\Mod\op,\S)}(\Emb_f(-,M),\Emb_f(-,N))\]
and similarly for $\Emb_f(P,N)$ and $\Emb_f(Q,N)$. The category $\on{Fun}(S_\tau^{\circlearrowright}\Mod\op,\S)$ is a symmetric monoidal model category in which fibrations and weak equivalences are objectwise.

In fact, more generally, if $\cat{A}$ is a small simplicial symmetric monoidal category, the category of simplicial functors to simplicial sets $\on{Fun}(\cat{A},\S)$ with the projective model structure and the Day tensor product is a symmetric monoidal model category (this is proved in \cite[Proposition 2.2.15]{isaacsoncubical}). It is easy to check that in this model structure, a representable functor is automatically cofibrant (this comes from the characterization in terms of lifting against trivial fibrations together with the fact that trivial fibration in $\S$ are epimorphisms).  Moreover, we have the identity
\[\Emb_f(-,M)\cong\Emb_f(-,P)\otimes\Emb_f(-,Q)\]
This immediatly implies that 
\[\Emb_f(-,P)\to\Emb_f(-,M)\]
is a cofibration in  $\on{Fun}(S_\tau^{\circlearrowright}\Mod\op,\S)$.
But the category  $\on{Fun}(S_\tau^{\circlearrowright}\Mod\op,\S)$ is also a model category enriched in $\S$, therefore, the induced map
\[\Map_{\on{Fun}(S_\tau^{\circlearrowright}\Mod\op,\S)}(\Emb_f(-,M),\Emb_f(-,N))\]
\[\to\Map_{\on{Fun}(S_\tau^{\circlearrowright}\Mod\op,\S)}(\Emb_f(-,P),\Emb_f(-,N))\]
is a fibration by the pushout-product property.
\end{proof}

Note that a linear embedding preserves the framing on the nose. Therefore, there is a well defined inclusion
\[\oper{KS}\to (S^1)^{\circlearrowright}_\tau\oMod\]

\begin{prop}
This map is a weak equivalence.
\end{prop}

\begin{proof}
There is a restriction map
\[S_\tau^{\circlearrowright}\oMod(a^{\boxplus n}\boxplus m;m)\to\Emb_f(D^{\sqcup n},S\times[0,1))\]
This map is a fibration by \ref{restriction fibration}. Its fiber over a particular configuration of disks is the space of embeddings of $S\times[0,1)$ into the complement of that configuration. By \ref{restriction to the boundary}, this space is weakly equivalent to $\Emb_f(S,S)$.

We have a diagram
\[
\xymatrix{
\Emb_{lin}^\partial(S^1\times[0,1)\sqcup D^n,S^1\times[0,1)\ar[d]\ar[r]&\Emb_f^\partial(S^1\times[0,1)\sqcup D^{\sqcup n},S^1\times[0,1))\ar[d]\\
\Emb_{lin}(D^{\sqcup n},S^1\times(0,1))\ar[r]&\Emb_f(D^{\sqcup n},S^1\times(0,1))
}
\]
Both vertical maps are fibrations. The bottom map is a weak equivalence since both sides are weakly equivalent to $\on{Conf}(n,S^1\times(0,1))$. To prove that the upper horizontal map is a weak equivalence, it suffices to check that it induces an equivalence on each fiber. The map induced on the fibers is weakly equivalent to the inclusion
\[S^1\to\Emb_f(S^1,S^1)\]
It is well-kown that this map is a weak equivalence.
\end{proof}

\section{Action of the higher version of $\oper{KS}$}

We are now ready to state and prove the main theorem of this paper

\begin{theo}\label{main}
Let $(B,A)$ be an algebra over the operad $\oper{E}_d^\partial$ in the category $\cat{C}$. Let $M$ be a framed $(d-1)$-manifold and $\tau$ be the product framing on $TM\oplus\mathbb{R}$. The pair $(B,\int_M A)$ is weakly equivalent to an algebra over the operad $M_\tau^{\circlearrowright}\oMod$.
\end{theo}

\begin{proof}
The construction $\int_{-}(B,A)$ is a simplicial functor $ f\Mfld_d^\partial\to\cat{C}$. Hence, $\int_{-}(B,A)$ is a functor from the full subcategory of $ f\Mfld_d^\partial$ spanned by disjoint unions of copies of $D$ and $M\times[0,1)$ to $\cat{C}$. Moreover this functor is symmetric monoidal. The operad $M_\tau^{\circlearrowright}\oMod$ has a map to the endomorphism operad of the pair $(D,M\times[0,1))$ in the symmetric monoidal category $f\Mfld_d^\partial$, therefore $(\int_D(B,A),\int_{M\times[0,1)}(B,A))$ is an algebra over $M_\tau^{\circlearrowright}$. To conclude, we use the fact that $\int_D(B,A)\cong B$ by Yoneda's lemma and $\int_{M\times[0,1)}(B,A)\simeq \int_M A$ by \ref{coro}.
\end{proof}

This theorem is mainly interesting because of the following theorem due to Thomas (see \cite{thomaskontsevich}):

\begin{theo}
Let $A$ be an $\oper{E}_{d-1}$-algebra in $\cat{C}$, then there is an algebra $(B',A')$ over $\oper{E}_d^\partial$ such that $B'$ is weakly equivalent to $\on{HH}_{\oper{E}_{d-1}}(A)$ and $A'$ is weakly equivalent to $A$.
\end{theo}

Combining these two results we have the following

\begin{coro}\label{mainbis}
We keep the notations of \ref{main}. The pair $(\on{HH}_{\oper{E}_{d-1}}(A),\int_M A)$ is weakly equivalent to an algebra over the operad $M_\tau^{\circlearrowright}\oMod$.
\end{coro}

The previous theorem has the following interesting corollary:

\begin{theo}
Let $(M,\tau)$ be a framed $(d-1)$-dimensional and $N$ be a $(d-1)$-connected manifold. The pair $(\Map(S^{d-1},N)^{-TN},\Sigma^{\infty}_+\Map(M,N))$ is weakly equivalent to an algebra over $M_\tau^{\circlearrowright}\oMod$.
\end{theo}

\begin{proof}
Let $R=\Sigma^{\infty}_+\Omega^dN$. $R$ is an $\oper{E}_d$-algebra in $\Spec$. It is proved in \cite{kleinfiber} that 
\[\HC{\oper{E}_d}(R)\simeq \Map(S^{d},N)^{-TN}\]
Similarly, it is proved in \cite{francisfactorization} that \[\int_MR\simeq\Sigma^{\infty}_+\Map(M,N)\]
The result is then a direct corollary of \ref{mainbis}.
\end{proof}

\begin{rem}
This result remains true if $N$ is a Poincaré duality space.
\end{rem}

\subsection*{Remark about the case of chain complexes}

It is desirable to have a version of our theorem when $\cat{C}$ is the category of unbounded chain complexes. If $R$ is a $\mathbb{Q}$-algebra, then the category $\on{Ch}_*(R)$ is a symmetric monoidal model category enriched over itself (but not a simplicial model category). Let us denote by $C_*(-;R)$ a lax symmetric monoidal functor from the category of topological spaces to the category of chain complexes over $R$ such that the homology of $C_*(X;R)$ is naturally isomorphic to the homology of $X$ with coefficients in $R$ (for instance one can take the singular chains over $R$. 

For $\oper{O}$ any operad in topological spaces, $C_*(\oper{O})$ is an operad in $\on{Ch}_*(R)$ and it has been shown by Hinich in \cite{hinichrectification} that the category of $C_*(\oper{O})$-algebra in $\on{Ch}_*(R)$ admits a transferred model structure.

For a $C_*(\oper{E}_d)$-algebra $A$ in $\on{Ch}_*(R)$, one can define factorization homology as the enriched coend
\[\int_MA:=C_*(\Emb^f(-,M))\otimes^{\L}_{C_*(\cat{E}_d)}A\]
and similarly in the case of manifolds with boundary. In the end one proves the following theorem exactly as \ref{main}.

\begin{theo}
Let $(B,A)$ be an algebra over the operad $C_*(\oper{E}_d^\partial)$ in the category $\on{Ch}_*(R)$. Let $M$ be a framed $(d-1)$-manifold and $\tau$ be the product framing on $TM\oplus\mathbb{R}$. The pair $(B,\int_M A)$ is weakly equivalent to an algebra over the operad $C_*(M_\tau^{\circlearrowright}\oMod)$.
\end{theo}

\begin{rem}
If $R$ is not a $\mathbb{Q}$-algebra, then the category of $C_*(\oper{O})$-algebras cannot necessarily be given the transferred model structure. It has been shown by Fresse in \cite{fressemodules} that there is still a left model structure. We are confident that up to minor modifications, our result remains true in this situation as well.
\end{rem}

\bibliographystyle{alpha}
\bibliography{biblio}

\end{document}